\documentclass[english,12pt]{article} 
\pdfoutput=1
\usepackage{authblk}
\usepackage[top=2cm,bottom=2cm,left=1cm,right=1cm]{geometry}
\usepackage{fontenc}
\usepackage[utf8]{inputenc}
\usepackage{afterpage}
\usepackage[dvipsnames]{xcolor}
\usepackage{newtxtext,commath}
\usepackage{amscd}
\usepackage{amsmath}
\usepackage{amsthm}
\usepackage{newtxmath}
\usepackage{pdfsync}
\usepackage{a4wide}
\usepackage{hyperref}
\usepackage{amsfonts}
\usepackage{color}
\usepackage{enumerate}
\usepackage{graphicx}
\usepackage{amsfonts}
\usepackage{ragged2e}

\newtheorem{theorem}{Theorem}[section]
\newtheorem{corollary}[theorem]{Corollary}
\newtheorem{lemma}[theorem]{Lemma}

\newtheorem{definition}[theorem]{Definition}
\newtheorem{remark}[theorem]{Remark}
\newtheorem{prop}[theorem]{Proposition}

\numberwithin{equation}{section}

\newcommand{\R}{\mathbb{R}}

\newcommand{\be}{\beta}
\newcommand{\eps}{\varepsilon}

\newcommand{\dvg}{\mathtt{div}}

\title{Best constants in bipolar $L^p$- Hardy-type Inequalities 	\thanks{This work was partially supported by the Romanian Ministry of Research, Innovation and Digitization,
		CNCS - UEFISCDI, project number PN-III-P1-1.1-TE-2021-1539, within PNCDI III. }}
\author[1, 2]{Cristian Cazacu} 
\author[1]{Teodor Rugină}

\affil[1]{Faculty of Mathematics and Computer Science and ICUB-The Research Institute of the University of Bucharest, 
	University of Bucharest,
		14 Academiei Street, 
		010014 Bucharest, Romania.\newline Emails: cristian.cazacu@fmi.unibuc.ro, teorugina@yahoo.com}
\affil[2]{Gheorghe Mihoc-Caius Iacob Institute of Mathematical
	Statistics and Applied Mathematics of the Romanian Academy\\
	050711 Bucharest, Romania}

\begin{document}

\maketitle


\begin{abstract}
In this work we prove sharp $L^p$ versions of the multipolar Hardy inequalities in \cite{cazacu1, cazacu3}, in the case of a bipolar potential and  $p\geq 2$. Our results are sharp and  minimizers do exist in the energy space. New features appear when $p>2$ compared to the linear case $p=2$ at the level of criticality of the p-Laplacian $-\Delta_p$ perturbed by a singular Hardy bipolar potential.  
\end{abstract}
\section{Introduction}
In the last decades, motivated by problems in quantum mechanics, there has been a consistent interest in Schr\"{o}dinger operators with multi-singular potentials and their applications to spectral theory and partial differential equations. It is well-known that qualitative properties of such operators are related to the validity of Hardy-type inequalities. A significant work has been done in the $L^2$-setting for linear Hamiltonians of the form $H:=-\Delta -W$ where $W$ denotes a potential with $n$ quadratic singular poles $a_i$, with $i=\overline{1,n}$, in the Euclidian space $\R^N$, $N\geq 3$. The most studied cases focus especially on $W^{(1)}=\sum_{i=1}^n \frac{\lambda_i}{|x-a_i|^2}$, $\lambda_i\in \R$, and $W^{(2)}=\sum_{1\leq i<j \leq n}\frac{|a_i-a_j|^2}{|x-a_i|^2|x-a_j|^2}$.  Various Hardy-inequalities were proved for $W^{(1)}$ and applied then to the well-posedness and asymptotic behaviour to some nonlinear elliptic equations in a series of papers by Felli-Terracini (e.g. \cite{felli1}, \cite{felli2}, \cite{felli3}) and more recently by Canale et al. in \cite{canale2, canale3, canale}, in the context of evolution problems involving Kolmogorov operators. As far as we know, the potential $W^{(2)}$ was firstly analyzed by Bosi-Dolbeault-Esteban in \cite{bosi} in order to determine lower bounds for the spectrum of some Schr\"{o}dinger and Dirac-type operators and later on by Cazacu et. al in \cite{cazacu1, cazacu3}. More precisely, in \cite{cazacu3} the authors proved the following multipolar Hardy inequality
\begin{equation}\label{sharp_CZ}
	\int_{\R^N} |\nabla u|^2 dx \geq \frac{(N-2)^2}{n^2}\int_{\R^N} W^{(2)}|u|^2 dx, \quad \forall u\in C_c^\infty(\R^N),     
\end{equation} 
with $N\geq 3$ and $n\geq 2$, where the constant $\frac{(N-2)^2}{n^2}$ in \eqref{sharp_CZ} is sharp and not achieved in the energy space 
$\mathcal{D}^{1, 2}(\R^N):=\Big\{u\in\mathcal{D}'(\R^N) \;\Big|\; \int_\R \abs{\nabla u}^2 dx < \infty \Big\}$.
The main goal here is to extend \eqref{sharp_CZ} to the $L^p$-setting. As far as we know such a result has not been obtained yet. For a one singular potential of the form $W^{(3)}=1/|x|^p$ there is a huge variety  of functional inequalities and applications around the celebrated $L^p$ Hardy inequality which holds for any $N>1$ and $1\leq p< N$ 
\begin{equation}\label{sharp_one_sing}
		\int_{\R^N} |\nabla u|^p dx \geq \left(\frac{N-p}{p}\right)^p\int_{\R^N} \frac
		{|u|^p}{|x|^p} dx, \quad \forall u\in C_c^\infty(\R^N),   
\end{equation} 
where the constant $\left(\frac{N-p}{p}\right)^p$ is sharp and not achieved (see, e.g.  \cite{balinsky}). In addition, no reminder terms can be added to the right hand side of \eqref{sharp_one_sing}. We recall the following definition from 
\cite{cazacu2}: 
\begin{definition}
    We say that $-\Delta_p$ is a subcritical operator if and only if there exists a non-negative potential $V\in L_{loc}^1(\R^N)$, $V\neq 0$, such that $-\Delta_p\geq V\abs{\cdot}^{p-2}\cdot$ in the sense of $L^2$ quadratic forms, that is, 
\begin{equation}
    \int_{\R^N} \abs{\nabla u}^p dx \geq \int_{\R^N} V \abs{u}^p dx, \;\;\forall u\in W^{1,p}(\R^N).  \notag
\end{equation}
Otherwise, we say that $-\Delta_p$ is critical.
\end{definition}

In view of  it is known that the nonlinear operator $-\Delta_p \cdot - \left(\frac{N-p}{p}\right)^p \frac{|\cdot|^{p-2}\cdot}{|x|^p}$ is critical. For a more extensive work on variants of $L^p-$Hardy inequalities one can consult for instance  \cite{balinsky}, \cite{tertikas} and references therein. \\
Our paper is structured as follows. In section 2 we state the main results and make some comments. In sections 3 and 4 we give the proofs of the main theorems, namely Theorem \ref{th1} and Theorem \ref{th2}, respectively.

 \section{Main results} 
To fix the hypotheses, let us consider $N\geq 3$, $1< p <N$ and two points $a_1$, $a_2\in\R^N$ arbitrarily fixed. Also, we denote by $a$ the medium point on the segment $[a_1,a_2]$, that is $a:=\frac{a_1+a_2}{2}$. We introduce the bipolar potentials $V_1$ and $V_2$ defined in $\R^N$ by 
\begin{equation}
    V_1(x) = \frac{\abs{a_1-a_2}^2\abs{x-a}^{p-2}}{\abs{x-a_1}^p\abs{x-a_2}^p}  \label{ec:potential_1}
\end{equation}
 respectively, 
\begin{equation}
	 V_2(x) =\frac{\abs{x-a}^{p-4}}{\abs{x-a_1}^p\abs{x-a_2}^p}\Big[\abs{x-a_1}^2\abs{x-a_2}^2 - \big( (x-a_1)\cdot(x-a_2) \big)^2\Big]. \label{potential_2}
\end{equation}
Notice that both $V_1$ and $V_2$ are non-negative, by Cauchy-Schwarz inequality. The potentials $V_1$ and $V_2$ blow-up at the singular poles $a_1$ and $a_2$. It is interesting that our new potentials involve also the medium point $a$. When $p>2$, $V_1$ degenerates at $a$, while the same is true for $V_2$ for $p>4$. Also, $V_1$ blows-up at $a$ for $p<2$. On the other hand, there is no $\gamma\in\R$ such that the limit $\lim_{x\to a} V_2(x)/\abs{x-a}^\gamma$ is finite. This can be justified by computing the limit across different directions (e.g. on the line segment $[a_1,a_2]$ and on the mediator of the segment).     

Let $\mu_1$ and $\mu_2$ be the constants depending on $N$ and $p$, defined as
\begin{align}
    & \mu_1= \frac{p-1}{4}\Bigg(\frac{N-p}{p-1}\Bigg)^p \;, \;\; \mu_2=\frac{p-2}{2}\Bigg( \frac{N-p}{p-1} \Bigg)^{p-1}.   \label{ec:constante}
\end{align}
Finally, denote by $V$ the following potential
\begin{equation}\label{sum_potential}
	V:=\mu_1 V_1 + \mu_2 V_2.    
\end{equation}
Note that 
\begin{prop}  \label{prop1}
It holds that $V\geq 0$ for any $1<p<N$. 
\end{prop}
For the sake of completeness we give a proof Proposition \ref{prop1} in the Appendix. 

We also define the functional space $\mathcal{D}^{1,p}(\R^N)$ as
\begin{equation}
    \mathcal{D}^{1,p}(\R^N):=\Big\{u\in\mathcal{D}'(\R^N) \;\Big|\; \int_\R \abs{\nabla u}^p dx < \infty \Big\}.
\end{equation}

Now we are in position to state our main results. 
\begin{theorem}\label{th1}
Let $N\geq 3$,  $1<p<N$ and $V$ as in \eqref{sum_potential}. For any $u\in C_c^\infty(\R^N)$ it holds
\begin{align}
     \int_{\R^N} \abs{\nabla u}^p dx 
     & \geq \int_{\R^N} V \abs{u}^p dx.   \label{ec:Hardy1}
\end{align}
Moreover, for $2\leq p<N$ the constant 1 is sharp in \eqref{ec:Hardy1} and it is achieved in the space $\mathcal{D}^{1,p}(\R^N)$ by the  minimizers of the form   
$$\phi(x)=\lambda\abs{x-a_1}^\frac{p-N}{2(p-1)} \abs{x-a_2}^\frac{p-N}{2(p-1)}, \lambda\in \R$$ 
unless the case $p=2$ when the best constant is not achieved.  
\end{theorem}
When restricted to the potential $V_1$ we have 
\begin{theorem}\label{th2}  For any $N\geq 3$ and  $2\leq p<N$, the following inequality holds
	\begin{align}
		\int_{\R^N} \abs{\nabla u}^p dx 
		& \geq \mu_1\int_{\R^N} V_1 \abs{u}^p dx, \quad \forall u \in C_c^\infty(\R^N),   \label{ec:Hardy2}
	\end{align}
Moreover, for any $2\leq p<N$ the constant $\mu_1$ is sharp in \eqref{ec:Hardy2}. but not achieved in $\mathcal{D}^{1,p}(\R^N)$. 
\end{theorem}
Let us emphasize some of the properties regarding our new potentials $V_1$ and $V_2$ and discuss how these new results generalize the work of previous authors.
\begin{remark}   \label{rmk1}
	\begin{enumerate}[1)]
	    \item Notice that the potential integrals $V_1$, $V_2$ $\in  L^1_{loc}(\R^N)$, for any $1<p<N$.   \label{rmk11}
	    \item The structure of the potential $V_1$ change significantly  when $1<p<N$ compared to the case $p=2$ due to a degeneracy (singularity) which appears at the middle point $a$.    \label{rmk14}
		\item The inequality \eqref{ec:Hardy1}  is an improvement of \eqref{ec:Hardy2} when $p>2$, they coincide  when $p=2$ with the result in \eqref{sharp_CZ}, while \eqref{ec:Hardy2} becomes stronger than \eqref{ec:Hardy1} when $1< p<2$.    \label{rmk12}
		\item   Theorem \ref{th2} establishes a clear generalization of \eqref{sharp_CZ} since $V_1=W^{(2)}$ and $\mu_1=\frac{(N-2)^2}{4}$ when $n=p=2$.    \label{rmk13}
		\end{enumerate}
\end{remark}
As a consequence of Theorems \ref{th1}-\ref{th2} we have a surprising result
\begin{corollary}
	The operator $-\Delta_p \cdot - \mu_1 V_1 |\cdot|^{p-2} \cdot $ is subcritical when $p>2$ in opposite with the case $p=2$ when it becomes critical.  
\end{corollary}

\section{Proof of Theorem \ref{th1}}
The proof of the first part of Theorem \ref{th1} relies partially on an adaptation of the method of supersolutions introduced by Allegretto and Huang in \cite{picone}.  To be more specific, we apply the following general result
\begin{prop} \label{prop2}
	Let $N\geq 3$, $1<p<\infty$. If there exists a  function $\phi>0$ such that $\phi \in C^2(\R^n\setminus \{a_1, \ldots, a_n\})$ and  
	\begin{equation}\label{ine1}
		-\Delta_p\phi \geq \mu V \phi^{p-1}, \quad \forall  x\in\R^N\setminus\{a_1,...,a_n\},           
	\end{equation}
	where $V>0$, with $V\in L_{loc}^1(\R^N)$, is a given multi-singular potential with the poles  $a_1, \ldots, a_n$, then
	\begin{equation}
		\int_{\R^N} \abs{\nabla u}^p dx \geq \mu\int_{\R^N} V\abs{u}^p dx, \;\; \forall u\in C_c^\infty(\R^N).     \notag
	\end{equation}
\end{prop} 
The proof of Proposition \ref{prop2} is a trivial consequence of  Theorem \ref{th1} proven in \cite{picone}.  
\subsection{Determination of the triplet $(\mu, \phi, V)$ in \eqref{ine1}}
In order to prove inequality \eqref{ec:Hardy1}  in Theorem \ref{th1} it is enough to show that, using notations introduced above,  
$(1, \phi, \mu_1V_1+\mu_2V_2)$ is an admissible triplet in Proposition \ref{prop2}.  Up to some technicalities, this could be checked by direct computations. 
However, in the following we will explain how we reach to this triplet. 
Therefore, we want to find a function $\phi>0$, a constant $\mu$  and a potential $V$, with singularities in the points $a_1$ and $a_2$, depending on $N$ and $p$, which satisfy the identity 
\begin{equation}
    -\frac{\Delta_p\phi}{\phi^{p-1}} = \mu V, \;\; a.e. \text{for}\; x\in\R^N\setminus\{a_1,a_2\}.        \notag
\end{equation}
Inspired by the paper \cite{cazacu3} in the case $p=2$,  for the general case $1<p<N$ we follow, up to some point, the same strategy by considering the functions
\begin{equation} 
\phi_i=\abs{x-a_i}^\be, i=1,2,   \notag
\end{equation}
where $\be$ is negative, aimed to depend on $N$ and $p$ that  will be precised later. We introduce 
\begin{equation}
   \phi=\phi_1 \phi_2.  \label{ec:defphi}
\end{equation}
We compute the $p$-Laplacian of $\phi$ in \eqref{ec:defphi} in several steps. First, we note that
\begin{equation}
    \nabla\phi=\Bigg( \frac{\nabla \phi_1}{\phi_1} + \frac{\nabla \phi_2}{\phi_2} \Bigg)\phi.   \notag
\end{equation}
To simplify the notation, denote 
\begin{equation}   
     v:=\frac{\nabla \phi_1}{\phi_1} + \frac{\nabla \phi_2}{\phi_2}.  \label{ec:defv}
\end{equation}
Using the definition of the $p$-Laplacian operator, we obtain
\begin{align}
    \Delta_p \phi & = \dvg\Big( \abs{\nabla\phi}^{p-2}\nabla\phi \Big) \notag \\
    & = \dvg\Big(\phi^{p-1} \abs{v}^{p-2}v  \Big) \notag \\
    & = \nabla\Big(\phi^{p-1}\Big)\abs{v}^{p-2}v + \phi^{p-1}\nabla\Big(\abs{v}^{p-2}\Big)v + \phi^{p-1}\abs{v}^{p-2} \dvg(v) .  \notag
\end{align}
Hence, we denote
\begin{equation}  
     -\frac{\Delta_p\phi}{\phi^{p-1}} =: V,    \notag
\end{equation}
where 
\begin{equation}
    V = -\Bigg[ \nabla\Big(\abs{v}^{p-2}\Big)\cdot v + \abs{v}^{p-2}\dvg(v) + (p-1)\abs{v}^p \Bigg] \;    \label{ec:potential}
\end{equation}
Next, we compute explicitly the three terms in (\ref{ec:potential}). The expression of $v$ in (\ref{ec:defv}) is given by 
\begin{equation}
     v = \Bigg(\frac{x-a_1}{\abs{x-a_1}^2}+\frac{x-a_2}{\abs{x-a_2}^2} \Bigg). 
\end{equation}
Taking the modulus, we get
\begin{equation}
     \abs{v} = 2\vert\be\vert \frac{\abs{x-a}}{\abs{x-a_1}\abs{x-a_2}},  \notag
\end{equation}
where $a=\frac{a_1+a_2}{2}$ is the medium point on the segment $[a_1,a_2]$.
The gradient in (\ref{ec:potential}) becomes
\begin{align}
   \nabla\Big( \abs{v}^{p-2} \Big) & = (p-2)\abs{v}^{p-3}\nabla\big( \abs{v} \big)  \notag\\
    & = (p-2)\abs{v}^{p-3} 2\abs{\beta} \Bigg[ \frac{x-a}{\abs{x-a}\abs{x-a_1}\abs{x-a_2}}-\frac{\abs{x-a}(x-a_1)}{\abs{x-a_1}^3\abs{x-a_2}}-\frac{\abs{x-a}(x-a_2)}{\abs{x-a_1}\abs{x-a_2}^3} \Bigg]   \notag\\
    & = (p-2)\abs{v}^{p-2}\Bigg[ \frac{x-a}{\abs{x-a}^2} - \frac{x-a_1}{\abs{x-a_1}^2} - \frac{x-a_2}{\abs{x-a_2}^2} \Bigg].  \notag \\
    & = (p-2)\abs{v}^{p-2}\Bigg[ \frac{x-a}{\abs{x-a}^2} -\frac{v}{\be} \Bigg].  \label{ec:term1}
\end{align}
The second term in \eqref{ec:potential} yields to
\begin{equation}
    \dvg(v) = \be (N-2)\Bigg[ \frac{1}{\abs{x-a_1}^2} + \frac{1}{\abs{x-a_2}^2} \Bigg].   \label{ec:term2}
\end{equation}
Using \eqref{ec:term1}, \eqref{ec:term2} and \eqref{ec:potential}, the potential reduces to
\begin{align}\label{ec:potential2}
    V  &= -\abs{v}^{p-2} \Bigg[ (p-2)\left( \frac{x-a}{\abs{x-a}^2}-\frac{1}{\be}v \right)\cdot v   \notag\\
    & \quad \quad \quad \quad \quad \quad + \be(N-2)\sum_{i=1}^2 \frac{1}{\abs{x-a_i}^2} + 4(p-1)\be^2\frac{\abs{x-a}^2}{\abs{x-a_1}^2\abs{x-a_2}^2} \Bigg]  \notag\\
    & = -\abs{v}^{p-2} \left[ T_1 + T_2 + T_3 \right],   
\end{align}
where we denoted 
\begin{align}
    T_1 & = (p-2)\left( \frac{x-a}{\abs{x-a}^2}-\frac{1}{\be}v \right)\cdot v,   \notag\\
    T_2 & = \be(N-2)\sum_{i=1}^2 \frac{1}{\abs{x-a_i}^2},  \notag\\
    T_3 & = 4(p-1)\be^2\frac{\abs{x-a}^2}{\abs{x-a_1}^2\abs{x-a_2}^2}.  \notag
\end{align}
Next, we rearrange the expression of $V$ in \eqref{ec:potential2}. To simplify the computations we employ the notations $v_1:=x-a_1$ and $v_2:=x-a_2$. Hence, $x-a=\frac{v_1+v_2}{2}$ and then we obtain  
\begin{align}
    \abs{v}^{p-2} & = 2^{p-2}\abs{\beta}^{p-2} \frac{\abs{x-a}^{p-2}}{\abs{x-a_1}^{p-2}\abs{x-a_2}^{p-2}}  \notag\\
    & = \abs{\beta}^{p-2} \frac{\abs{v_1+v_2}^{p-2}}{\abs{v_1}^{p-2}\abs{v_2}^{p-2}}.        \label{ec:t0}
\end{align}
The first term in \eqref{ec:potential2} becomes 
\begin{align}
    T_1 & = (p-2)\Bigg( \frac{x-a}{\abs{x-a}^2}-\frac{1}{\be}v \Bigg)\cdot v \notag\\
    & = (p-2)\frac{v_1+v_2}{2}\frac{4}{\abs{v_1+v_2}^2}\beta\Bigg(\frac{v_1}{\abs{v_1}^2}+\frac{v_2}{\abs{v_2}^2}\Bigg) - 4(p-2)\beta\frac{\abs{v_1+v_2}^4}{4\abs{v_1}^2\abs{v_2}^2}   \notag\\
    & = (p-2)\beta \frac{2(v_1+v_2)(v_1\abs{v_2}^2+v_2\abs{v_1}^2)-\abs{v_1+v_2}^2}{\abs{v_1}^2\abs{v_2}^2\abs{v_1+v_2}^2}  
    \label{ec:t1}
\end{align}
The second term in \eqref{ec:potential2} reads to 
\begin{align}
    T_2 & = \be(N-2)\sum_{i=1}^2 \frac{1}{\abs{x-a_i}^2}   \notag\\
    & = \beta(N-2)\frac{\abs{v_1}^2+\abs{v_2}^2}{\abs{v_1}^2\abs{v_2}^2}.   \label{ec:t2}
\end{align}
The third term in \eqref{ec:potential2} is 
\begin{align}
    T_3 & = 4(p-1)\be^2\frac{\abs{x-a}^2}{\abs{x-a_1}^2\abs{x-a_2}^2}   \notag\\
    & = (p-1)\beta^2\frac{\abs{v_1+v_2}^2}{\abs{v_1}^2\abs{v_2}^2}.  \label{ec:t3}
\end{align}
From (\ref{ec:t0}), (\ref{ec:t1}), (\ref{ec:t2}) and (\ref{ec:t3}) we get successively

\begin{align}\label{id2}
    V & = -\abs{\beta}^{p-2}\beta \frac{\abs{v_1+v_2}^{p-2}}{\abs{v_1}^{p-2}\abs{v_2}^{p-2}}\times\frac{1}{\abs{v_1}^2\abs{v_2}^2\abs{v_1+v_2}^2} \Bigg[ (N-2)(\abs{v_1}^2+\abs{v_2}^2)\abs{v_1+v_2}^2+ \notag\\
    & \quad \quad \quad \quad +(p-1)\beta\abs{v_1+v_2}^4 +  2(p-2)(v_1+v_2)(v_1\abs{v_2}^2+v_2\abs{v_1}^2)-(p-2)\abs{v_1+v_2}^4 \Bigg]  \notag\\
    = & \abs{\be}^{p-1} \frac{\abs{v_1+v_2}^{p-4}}{\abs{v_1}^p\abs{v_2}^p}
    \Bigg[ \Big( (p-1)\beta+N-p\Big)(\abs{v_1}^4+\abs{v_2}^4) + 2\Big(2(p-1)\beta+N-p\Big)v_1\cdot v_2(\abs{v_1}^2+\abs{v_2}^2)  \notag\\
    & \quad \quad \quad \quad + 2\Big((p-1)\beta+N+p-4\Big)\abs{v_1}^2\abs{v_2}^2+ 4\Big((p-1)\beta+2-p)\Big)(v_1\cdot v_2)^2 \Bigg]  
\end{align}
In order to get rid of the cross term in the last identity of \eqref{id2} we choose $\be=\frac{p-N}{2(p-1)}$ which implies 
\begin{equation}\label{id3}
	\phi=\phi_1 \phi_2 = \abs{x-a_1}^\frac{p-N}{2(p-1)} \abs{x-a_2}^\frac{p-N}{2(p-1)}.
\end{equation}
Then, using notations above, we  get the following form of the potential $V$:
\begin{align}\label{id1}
    V & = \Bigg( \frac{N-p}{2(p-1)} \Bigg)^{p-1} \frac{\abs{v_1+v_2}^{p-4}}{\abs{v_1}^p\abs{v_2}^p}
    \Bigg[ \Big(\frac{N-p}{2}\Big)(\abs{v_1}^4+\abs{v_2}^4) +\notag\\ 
    & \quad \quad \quad \quad +\Big((N-p+4(p-2)\Big)\abs{v_1}^2\abs{v_2}^2 + \Big( 2(p-N)+4(2-p) \Big) (v_1\cdot v_2)^2   \Bigg]    \notag\\
    & = \Bigg( \frac{N-p}{2(p-1)} \Bigg)^{p-1} \frac{\abs{v_1+v_2}^{p-4}}{\abs{v_1}^p\abs{v_2}^p}
    \Bigg[ \Big(\frac{N-p}{2}\Big)(\abs{v_1}^4+\abs{v_2}^4) + (N-p)\abs{v_1}^2\abs{v_2}^2 \notag\\
    & \quad \quad \quad \quad- 2(N-p)(v_1\cdot v_2)^2 + 4(p-2)\Big( \abs{v_1}^2\abs{v_2}^2 - (v_1\cdot v_2)^2 \Big) \Bigg]  \notag\\
    & = (p-1)\Bigg( \frac{N-p}{2(p-1)} \Bigg)^p \frac{\abs{v_1+v_2}^{p-2}\abs{v_1-v_2}^2}{\abs{v_1}^p\abs{v_2}^p}\notag\\
    & \quad \quad \quad \quad +4(p-2)\Bigg( \frac{N-p}{2(p-1)} \Bigg)^{p-1} \frac{\abs{v_1+v_2}^{p-4}}{\abs{v_1}^p\abs{v_2}^p}\Big[ \abs{v_1}^2\abs{v_2}^2 - (v_1\cdot v_2)^2 \Big]       
    \end{align}
Undoing the notations $v_1$ and $v_2$ in \eqref{id1} we finally obtain 
\begin{multline} 
	V    = \frac{p-1}{4}\Bigg(\frac{N-p}{p-1}\Bigg)^p \frac{\abs{a_1-a_2}^2\abs{x-a}^{p-2}}{\abs{x-a_1}^p\abs{x-a_2}^p} 
  +   \frac{p-2}{2}\Bigg( \frac{N-p}{p-1} \Bigg)^{p-1} \times \\
  \times   \frac{\abs{x-a}^{p-4}}{\abs{x-a_1}^p\abs{x-a_2}^p}\Big[\abs{x-a_1}^2\abs{x-a_2}^2 - \big( (x-a_1)\cdot(x-a_2) \big)^2\Big]    \label{ec:potentialfinal}
\end{multline}
By \eqref{ec:potential_1}-\eqref{ec:constante} we can write
\begin{equation}
    V = \mu_1 V_1 + \mu_2 V_2         \notag
\end{equation} and $\phi$ in \eqref{id3} verifies the identity 
$$-\frac{\Delta_p \phi}{\phi^{p-1}}=V, \textrm{ in } \R^N\setminus\{a_1, a_2\}. $$
The proof of \eqref{ec:Hardy1} is complete now, by \ref{prop2}.

\subsection{Sharpness of inequality $(\ref{ec:Hardy1})$}

We want to show that, for $2<p<N$, the constant $\mu=1$ is sharp in inequality
\begin{equation}
    \int_{\R^N} \abs{\nabla u}^p dx \geq \mu\int_{\R^N} V\abs{u}^p dx,  \notag
\end{equation}
for $u\in C_c^\infty(\R^N)$ and it is actually attained in $\mathcal{D}^{1,p}(\R^N)$ by the function $\phi$ in \eqref{id3}.
We show that $\phi$ satisfies the identity:
\begin{equation}
    \int_{\R^N} \abs{\nabla \phi}^p dx = \int_{\R^N} V\abs{\phi}^p dx,  \label{eq:egalphi}
\end{equation}
which proves both of the facts stated above. This is done using integration by parts, but we need to check the integrability of   $|\nabla\phi|^p$. 
\begin{prop}
For $N\geq 3$, $2<p<N$ and $\phi$ in \eqref{id3} it holds that $\phi\in\mathcal{D}^{1,p}(\R^N)$. 
\end{prop}
\begin{proof}
Recall that $\be:=\frac{p-N}{2(p-1)}$ and
$$	\phi=\phi_1 \phi_2 = \abs{x-a_1}^\frac{p-N}{2(p-1)} \abs{x-a_2}^\frac{p-N}{2(p-1)}.$$
By direct computation we formally obtain,
\begin{equation}
	\nabla \phi = \be \abs{x-a_1}^{\frac{p-N}{2(p-1)}-2}\abs{x-a_2}^{\frac{p-N}{2(p-1)}-2} \Bigg[ \abs{x-a_2}^2(x-a_1) + \abs{x-a_1}^2(x-a_2) \Bigg].   \label{ec:gradphi}   
\end{equation}
By squaring the relation $(\ref{ec:gradphi})$, we get 
\begin{align}
	\abs{\nabla\phi}^2 &= 4\beta^2 \abs{x-a_1}^{2\big(\frac{p-N}{2(p-1)}-1\big)}\abs{x-a_2}^{2\big(\frac{p-N}{2(p-1)}-1\big)} \abs{x-a}^2.   \notag
\end{align}
Hence, 
\begin{align}
	\abs{\nabla\phi} &= 2\abs{\beta} \abs{x-a_1}^{\frac{p-N}{2(p-1)}-1}\abs{x-a_2}^{\frac{p-N}{2(p-1)}-1} \abs{x-a}.    \notag
\end{align}
Let $0<r<\frac{\abs{a_1-a_2}}{4}$ and $R\geq \max\Big\{2\abs{a_1}, 2\abs{a_2}\Big\}+2r$. Define $B_r^i:=B(a_i,r)$ to be the ball centered at $a_i$ and of radius $r$, $i=1,2$, and $B_R:=B(a,R)$, where $a=\frac{a_1+a_2}{2}$, to be the ball centered at $a$ and of radius $R$. By the choice of $r$ and $R$, we can see that $B_R$ contains both $B_r^1$ and $B_r^2$. We prove the $L^p$-integrability of $\nabla\phi$, for $2<p<N$, as follows. We split
\begin{align}
	\int_{\R^N} \abs{\nabla \phi}^p dx 
	&= \int_{\R^N\setminus B_R}\abs{\nabla\phi}^p dx + \int_{B_R}\abs{\nabla\phi}^p dx  \; := \; I_1+I_2.  \notag
\end{align}
First, we notice that, in $\R^N\setminus B_R$,
\begin{equation}
	\abs{x-a_i} < \abs{x}+\abs{a_i} < \abs{x} + R < 2\abs{x}, \label{ec:inegg1}
\end{equation}
\begin{equation}
	\abs{x-a_i} > \abs{x} - \abs{a_i} >\abs{x} - \frac{R}{2} = \frac{\abs{x}}{2} +\frac{\abs{x}-R}{2} > \frac{\abs{x}}{2}. \label{ec:inegg2}
\end{equation}
Similarly, we have that 
\begin{equation}
	\abs{x} \leq \abs{x-a} \leq 2\abs{x}  \;\; \text{in}\; \R^N\setminus B_R.   \label{ec:inegg3}
\end{equation}
Therefore, \;\;$\abs{x-a}$\;\; and \;\;$\abs{x-a_i}$\;\; behave asymptotically as \;\;$\abs{x}$, for $x\in \R^N\setminus B_R$.
Next we will write $"\simeq"$ and $"\lesssim"$ instead of usual notations, meaning that the equality or inequality holds up to some constant. By ($\ref{ec:inegg1}$), ($\ref{ec:inegg2}$), ($\ref{ec:inegg3}$) and co-area formula, we get
\begin{align}
	I_1 & \lesssim  \int_{\R^N\setminus B_R} \abs{x}^{\frac{(p-N)p}{2(p-1)}-p} \abs{x}^{\frac{(p-N)p}{2(p-1)}-p} \abs{x}^p dx   \notag\\
	& =  \int_{\R^N\setminus B_R} \abs{x}^{\frac{(p-N)p}{p-1}-p} dx   \notag\\
	& =  \int_{\R^N\setminus B_R} \abs{x}^{\frac{p(1-N)}{p-1}} dx   \notag\\
	& \simeq  \int_R^\infty s^{\frac{p(1-N)}{p-1}} s^{N-1} ds  \notag\\
	& =  \int_R^\infty s^\frac{1-N}{p-1} ds,      \notag
\end{align}
which is finite when $\frac{1-N}{p-1}+1<0$. So, $I_1<\infty$ for any $p<N$.\\ 
We now estimate $I_2$ by splitting it in two terms. 
\begin{equation}
	I_2 = \int_{B_r^1\cup B_r^2} \abs{\nabla\phi}^p dx +  \int_{B_R\setminus (B_r^1\cup B_r^2)} \abs{\nabla\phi}^p dx.   \notag
\end{equation}
The integral on $\overline{B_R\setminus (B_r^1\cup B_r^2)}$ is finite, since the function under integration is continuous. On the other hand, we have:
\begin{align}
	\int_{B_r^i} \abs{\nabla\phi}^p dx & \lesssim  \int_{B_r^i} \abs{x-a_i}^{{\frac{(p-N)p}{2(p-1)}}-p} dx   \notag\\
	& \simeq  \int_0^r s^{{\frac{(p-N)p}{2(p-1)}}-p} s^{N-1} ds   \notag\\
	& =  \int_0^r s^{\frac{(p-N)(2-p)}{2(p-1)}-1} ds.    \notag
\end{align}
For $2<p<N$, the integral above is finite, for $i=1,2$. In conclusion, $\phi$ belongs to $\mathcal{D}^{1,p}(\R^N)$ for any $2<p<N$.
\end{proof}
Taking into account that $V=-\frac{\Delta_p\phi}{\phi^p}$ and integrating by parts in \eqref{eq:egalphi}, we get
\begin{equation}
    \int_{\R^N} V\abs{\phi}^p dx = \int_{\R^N} -\frac{\Delta_p\phi}{\phi^{p-1}}\phi^p dx = \int_{\R^N} \dvg\Big(\abs{\nabla\phi}^{p-2}\nabla\phi\Big)\phi dx  = \int_{\R^N} \abs{\nabla \phi}^p dx,
\end{equation}
which concludes the proof of Theorem \ref{th1}.

\section{Proof of Theorem \ref{th2}}

First we need the following lemma.
\begin{lemma}  \label{lemma1}  
Assume $p\geq 2$. Let $\phi$ be a positive function in $\R^N$ with $\phi\in C^2(\R^N\setminus\{a_1,\;a_2\})$ and let $V\in L^1_{loc}(\R^N)$ be a continuous potential on $\R^N\setminus\{a_1,\;a_2\}$ such that 
\begin{equation}
    -\Delta_p \phi(x) - V \phi(x)^{p-1} \geq 0\;, \;\;\; \forall x\in\R^N\setminus\{a_1,\;a_2\}.   \label{ec:ip1}
\end{equation}
Then there exists $c_1(p)$ such that
\begin{equation}
   c_1(p) \int_{\R^N} \abs{\nabla\Big(\frac{u}{\phi}\Big)}^p \phi^p dx \leq \int_{\R^N} \abs{\nabla u}^p dx - \int_{\R^N} V\abs{u}^p dx \;,\;\; \forall \;u\in  C_c^\infty(\R^N\setminus\{a_1,\;a_2\}).   \label{ec:eq11}
\end{equation}
Moreover, assume we have equality in \eqref{ec:ip1}. Then the following reverse inequality holds for any $\;u\in  C_c^\infty(\R^N\setminus\{a_1,\;a_2\})$:
\begin{equation}
    \int_{\R^N} \abs{\nabla u}^p dx - \int_{\R^N} V\abs{u}^p dx  
    \leq \frac{p(p-1)}{2} \int_{\R^N} \Big(\phi\abs{\nabla \Big(\frac{u}{\phi}\Big) }+\frac{u}{\phi}\abs{\nabla\phi}\Big)^{p-2} \phi^2\abs{\nabla \Big(\frac{u}{\phi}\Big)}^2.   \label{ec:lemma322}
\end{equation}
\end{lemma}
\begin{proof}
The proof of inequality (\ref{ec:eq11}) is a trivial adaptation of Theorem 2.2 from \cite{cazacu2}. We focus now on the proof of \eqref{ec:lemma322}\\
Using the hypothesis and integrating by parts, we get
\begin{align}
    \int_{\R^N} \abs{\nabla u}^p & dx 
    - \int_{\R^N} V\abs{u}^p dx   = \int_{\R^N} \abs{\nabla u}^p dx + \int_{\R^N} \frac{\Delta_p\phi}{\phi^{p-1}}\abs{u}^p dx  \notag\\
    & = \int_{\R^N} \abs{\phi\nabla \Big(\frac{u}{\phi}\Big) + \frac{u}{\phi}\nabla\phi}^p - p \Big(\frac{u}{\phi}\Big)^{p-1}\phi\abs{\nabla\phi}^{p-2}\nabla \Big(\frac{u}{\phi}\Big)\cdot\nabla\phi - \Big(\frac{u}{\phi}\Big)^p\abs{\nabla\phi}^p dx  \label{ec:eq12}
\end{align}
We employ an  inequality from \cite{shafrir}: for $p\geq 2$ it holds
\begin{equation}
    \abs{x+y}^p - p \abs{y}^{p-1}y\cdot x - \abs{y}^p \leq \frac{p(p-1)}{2}\Big(\abs{x}+\abs{y}\Big)^{p-2} \abs{x}^2\;,\;\; \forall\;x,\;y\in\R^N.    \label{ec:ineqshafrir}
\end{equation}
Applying \eqref{ec:ineqshafrir} for $x=\phi\nabla \Big(\frac{u}{\phi}\Big)$ and $y=\frac{u}{\phi}\nabla\phi$ in ($\ref{ec:eq12}$) we obtain
\begin{align}
    \int_{\R^N} \abs{\nabla u}^p dx - \int_{\R^N} V\abs{u}^p dx  
    & \leq \frac{p(p-1)}{2} \int_{\R^N} \Big(\phi\abs{\nabla \Big(\frac{u}{\phi}\Big)}+\frac{u}{\phi}\abs{\nabla\phi}\Big)^{p-2} \phi^2
    \abs{\nabla \Big(\frac{u}{\phi}\Big)}^2.   \label{ec:eq13}  
\end{align}
\end{proof}
We can easily extend the result above to  functions $u$ in $W_0^{1, p}(\R^N)$.

\subsection{Asymptotic behavior of $V_1$ and $V_2$}
This section is also useful for the proof of optimality of $\mu_1$ in \eqref{ec:Hardy2}. 
In order to compare the potentials $V_1$ and $V_2$ we analyze their behavior at the singular points $a_1$, $a_2$, at the degenerate point $a$ and at infinity, respectively. Recall that
\begin{equation}
	V_1 = \frac{\abs{a_1-a_2}^2\abs{x-a}^{p-2}}{\abs{x-a_1}^p\abs{x-a_2}^p}.     \notag
\end{equation}
Fix $p$ such that $2<p<N$. Then one can easily see that
\begin{equation}
	\lim_{x\to a_i} \abs{x-a_i}^pV_1 = 2^{2-p}=:c_1.  \notag
\end{equation}
In the middle point $a=\frac{a_1+a_2}{2}$,\; $V_1$ tends to $0$, as 
\begin{equation}
	\lim_{x\to a} \abs{x-a}^{2-p} V_1 = 4^p\abs{a_1-a_2}^{2(1-p)}=:c_2. \notag
\end{equation}
At infinity, we have
\begin{equation}
	\lim_{\abs{x}\to\infty} \abs{x}^{p+2} V_1 = \abs{a_1-a_2}^2=:c_3. \notag  
\end{equation}
In consequence, for $p>2$ and $i=1,2$, we have 
\begin{equation}
	V_1(x)= \left\{\begin{array}{cc}
		c_1|x-a_i|^{-p},  & \textrm{ as } x\rightarrow a_i\notag\\[7pt]
		c_2|x-a|^{p-2}, & \textrm{ as } x\rightarrow a\notag\\[5pt]
		c_3|x|^{-(p+2)}, & \textrm{ as } x\rightarrow \infty \notag\\[5pt]	
	\end{array}\right.
\end{equation}
Now we look at $V_2$:
\begin{equation}
	V_2(x) =\frac{\abs{x-a}^{p-4}}{\abs{x-a_1}^p\abs{x-a_2}^p}\Big[\abs{x-a_1}^2\abs{x-a_2}^2 - \big( (x-a_1)\cdot(x-a_2) \big)^2\Big].    \notag
\end{equation}

Taking into account the asymptotic behavior of $V_1$ around the singularities $a_i$, we further emphasize that $V_2$ is dominated by $V_1$ in a neighbourhood of $a_i$.
\begin{prop}  \label{prop41}
There exists $r_0>0$ such that, for any $\delta>0$, it holds
\begin{equation}
    \mu_2 V_2 - \frac{\delta}{2}V_1 < 0.
\end{equation}
\end{prop}
\begin{proof}
Let $\delta>0$. Denote by $M:=\abs{a_1-a_2}$ and $\alpha:=\cos(\varphi)\in[-1,1]$, where $\cos(\varphi)=\frac{(x-a_1)\cdot(x-a_2)}{\abs{x-a_1}\abs{x-a_2}}$. Let $r_0>0$ aimed to be small and $x\in B(a_i,r_0)$. Then 
\begin{align}
     \mu_2 V_2 - \frac{\delta}{2}V_1 & < 0    \notag\\
\iff  \mu_2 \abs{x-a}^{p-4}  \Big[\abs{x-a_1}^2\abs{x-a_2}^2 & -\big((x-a_1)\cdot(x-a_2)\big)^2\Big]  - \frac{\delta}{2}\abs{x-a}^{p-2}\abs{a_1-a_2}^2 < 0 \notag\\
\iff 2\mu_2(1-\alpha^2)\abs{x-a_1}^2\abs{x-a_2}^2 & - \delta \abs{x-a}^2\abs{a_1-a_2}^2 < 0   \notag\\
\iff 8\mu_2(1-\alpha^2)r_0^2(r_0+M)^2 & - \delta M^4 <0.  \notag
\end{align}
When $r_0\to 0$, the left hand-side tends to $-\delta M^4 < 0$, which proves our statement.
\end{proof}

\subsection{Proof of Theorem \ref{th2}}
Finally, we are ready to prove that it holds
\begin{equation}
    \int_{\R^N} \abs{\nabla u}^p dx \geq \mu_1 \int_{\R^N} V_1\abs{u}^p   \notag
\end{equation}
for any $2\leq p<N$ and that the constant $\mu_1$ is sharp. \\
The inequality follows from Theorem \ref{th1} and Remark \ref{rmk1}. We will prove here the sharpness of the constant. The case $p=2$ is proved in \cite{cazacu1}, so here we prove it  for $p>2$. In order to do this, assume there exists $\eps_0>0$ such that
\begin{equation}
    \int_{\R^N} \abs{\nabla u}^p dx - (\mu_1+\eps_0) \int_{\R^N} V_1\abs{u}^p dx \geq 0, \;\;\; \forall u\in C_c^\infty(\R^N\setminus\{a_1, a_2\}).  \label{ec:contradictie}
\end{equation}
Denote 
\begin{equation}
    L[u]:=\int_{\R^N} \abs{\nabla u}^p dx - (\mu_1+\eps_0)\int_{\R^N} V_1\abs{u}^p dx.  \notag
\end{equation}
We add and subtract $\mu_2 V_2\abs{u}^p$ in the second integral:
\begin{align}
    L[u] & =\int_{\R^N} \abs{\nabla u}^p dx - \int_{\R^N} \big(\mu_1 V_1+\mu_2 V_2\big)\abs{u}^p dx + \int_{\R^N} \big(\mu_2 V_2 - \eps_0 V_1\big)\abs{u}^p dx  \notag\\
    & = \int_{\R^N} \abs{\nabla u}^p dx - V\abs{u}^p dx + \int_{\R^N} \big(\mu_2 V_2 - \eps_0 V_1\big)\abs{u}^p dx  \label{ec:eq21}
\end{align}
By Proposition \ref{prop41}, for $\delta=\eps_0$, we get that
\begin{equation}
    \mu_2 V_2 - \eps_0 V_1 < -\frac{\eps_0}{2}V_1 < 0.   \label{ec:eq22}
\end{equation}

For any $\eps>0$ chosen under the assumption that $\eps< \min\{\frac{1}{2}, r_0^2\}$, define the following cut-off function:
\[ \theta_\eps(x)=\left\{\begin{array}{l l}
    0, & \text{ if } x\in B_{\eps^2}(a_i), \;\;\; \text{for}\;\; i=1,2  \\
    \frac{\log\big(\abs{x-a_i}/\eps^2\big)}{\log\frac{1}{\eps}} , & \text{ if } x\in B_\eps(a_i)\setminus B_{\eps^2}(a_i), \;\;\; \text{for}\;\; i=1,2 \\
    \frac{\log\big(\eps/\abs{x-a_i}^2\big)}{\log\frac{1}{\eps}} , & \text{ if } x\in B_{\eps^{1/2}}(a_i)\setminus B_\eps(a_i) , \;\;\; \text{for}\;\; i=1,2 \\
    0, & \text{otherwise}.
    \end{array}\right.
\]
Consider $u_\eps=\phi\theta_\eps$, where $\phi$ is defined in \eqref{ec:defphi}. Due to the fact that $\theta_\eps\in W_0^{1, p}(\R^N)$ we conclude that $u_\eps$ is also in $W_0^{1, p}(\R^N)$. 
Taking $u=u_\eps$ in \eqref{ec:eq21} we get
\begin{align}
    L[u_\eps] & = \int_{\R^N} \abs{\nabla u_\eps}^p - V\abs{u_\eps}^p dx + \int_{\R^N} \big(\mu_2 V_2 - \eps_0 V_1\big)\abs{u_\eps}^p dx   \label{ec:eq23}\\
    & =: I_\eps + J_\eps.     \notag
\end{align}
We will use Lemma \ref{lemma1} in order to estimate $I_\eps$. By direct computation, the gradient of $\theta_\eps$ is
\[ \nabla\theta_\eps(x)=\left\{\begin{array}{l l}
    0, & \text{ if } x\in B_{\eps^2}(a_i), \;\;\; \text{for}\;\; i=1,2  \\
    \big(\log\frac{1}{\eps}\big)^{-1}\frac{x-a_i}{\abs{x-a_i}^2} , & \text{ if } x\in B_\eps(a_i)\setminus B_{\eps^2}(a_i), \;\;\; \text{for}\;\; i=1,2 \\
    -2\big(\log\frac{1}{\eps}\big)^{-1}\frac{x-a_i}{\abs{x-a_i}^2} , & \text{ if } x\in B_{\eps^{1/2}}(a_i)\setminus B_\eps(a_i) , \;\;\; \text{for}\;\; i=1,2 \\
    0, & \text{otherwise}.
    \end{array}\right.
\]
Restricting to the support of $\theta_\eps$, we get
\begin{align}
    I_\eps &= \sum_{i=1}^2\int_{B_\eps(a_i)\setminus B_{\eps^2}(a_i)} \abs{\nabla u_\eps}^p - V\abs{u_\eps}^p dx + \sum_{i=1}^2\int_{B_{\eps^{1/2}}(a_i)\setminus B_\eps(a_i)} \abs{\nabla u_\eps}^p - V\abs{u_\eps}^p dx =: I_{1,\eps} + I_{2,\eps}.   \notag
\end{align}
Recall that
\begin{equation}
    \phi(x)=\abs{x-a_1}^\be\abs{x-a_2}^\be \;,\;\;\; \be=\frac{p-N}{2(p-1)}.  \notag
\end{equation}
By \eqref{ec:lemma322} in Lemma \ref{lemma1} and using the co-area formula, we get the estimate
\begin{align}
    I_{1,\eps} & = \sum_{i=1}^2\int_{B_\eps(a_i)\setminus B_{\eps^2}(a_i)} \abs{\nabla u_\eps}^p - V\abs{u_\eps}^p dx     \notag\\
    & \lesssim \sum_{i=1}^2\int_{B_\eps(a_i)\setminus B_{\eps^2}(a_i)} \Big( \phi\abs{\nabla\theta_\eps}+\theta_\eps\abs{\nabla\phi} \Big)^{p-2}\phi^2\abs{\nabla\theta_\eps}^2 dx     \notag\\
    & \lesssim  \sum_{i=1}^2\int_{B_\eps(a_i)\setminus B_{\eps^2}(a_i)} \Big( \big(\log\frac{1}{\eps}\big)^{-1}\abs{x-a_i}^{\be-1} + \big(\log\frac{1}{\eps}\big)^{-1}\log\frac{\abs{x-a_i}}{\eps^2}\abs{x-a_i}^{\be-1} \Big)^{p-2}\times\nonumber\\
   & \times 
     \abs{x-a_i}^{2\be-2}\big(\log\frac{1}{\eps}\big)^{-2}    dx     \notag\\         
    & \simeq \Big(\log\frac{1}{\eps}\Big)^{-p} \sum_{i=1}^2\int_{B_\eps(a_i)\setminus B_{\eps^2}(a_i)} \Big(1 + \log\frac{\abs{x-a_i}}{\eps^2} \Big)^{p-2}\abs{x-a_i}^{p(\be-1)} dx     \notag\\
    & \simeq \Big(\log\frac{1}{\eps}\Big)^{-p} \int_{\eps^2}^\eps \Big( 1+\log\frac{s}{\eps^2} \Big)^{p-2}  s^{p(\be-1)+N-1} ds    \notag\\
    &\lesssim \Big(\log\frac{1}{\eps}\Big)^{-p} \Big(1+\log\frac{1}{\eps}\Big)^{p-2}  \int_{\eps^2}^\eps  s^{(p-N)\big( \frac{p}{2(p-1)}-1\big)-1} ds   \notag\\
    & \lesssim \Big(\log\frac{1}{\eps}\Big)^{-2} \eps^{(p-N)\big( \frac{p}{2(p-1)}-1\big)}.    \notag
\end{align}
Similarly,
\begin{align}
    I_{2,\eps} & = \sum_{i=1}^2\int_{B_{\eps^{1/2}}(a_i)\setminus B_\eps(a_i)} \abs{\nabla u_\eps}^p - V\abs{u_\eps}^p dx   \notag\\
    & \lesssim \sum_{i=1}^2\int_{B_{\eps^{1/2}}(a_i)\setminus B_\eps(a_i)} \Big( \phi\abs{\nabla\theta_\eps}+\theta_\eps\abs{\nabla\phi} \Big)^{p-2}\phi^2\abs{\nabla\theta_\eps}^2 dx     \notag\\ 
    & \lesssim  \sum_{i=1}^2\int_{B_\eps(a_i)\setminus B_{\eps^2}(a_i)} \Big( \big(\log\frac{1}{\eps}\big)^{-1}\abs{x-a_i}^{\be-1} + \big(\log\frac{1}{\eps}\big)^{-1}\log\frac{\eps}{\abs{x-a_i}^2}\abs{x-a_i}^{\be-1} \Big)^{p-2}
    \times\nonumber\\
    &\times \abs{x-a_i}^{2\be-2}\big(\log\frac{1}{\eps}\big)^{-2}    dx     \notag\\
    & \simeq \Big(\log\frac{1}{\eps}\Big)^{-p} \sum_{i=1}^2\int_{B_{\eps^{1/2}}(a_i)\setminus B_\eps(a_i)} \Big(1 + \log\frac{\eps}{\abs{x-a_i}^2} \Big)^{p-2} \abs{x-a_i}^{p(\be-1)} dx    \notag\\
    & \simeq \Big(\log\frac{1}{\eps}\Big)^{-p} 2\int_\eps^{\eps^{1/2}} \Big( 1+\log\frac{\eps}{s^2} \Big)^{p-2}  s^{p(\be-1)+N-1} ds    \notag\\
    &\lesssim \Big(\log\frac{1}{\eps}\Big)^{-p} \Big(1+\log\frac{1}{\eps} \Big)^{p-2} \int_\eps^{\eps^{1/2}}  s^{(p-N)\big( \frac{p}{2(p-1)}-1\big)-1} ds   \notag\\
    & \lesssim \Big(\log\frac{1}{\eps}\Big)^{-2} \eps^{(p-N)\big( \frac{p}{2(p-1)}-1\big)}.    \notag
\end{align}
Hence, the estimate for $I_\eps$ is
\begin{equation}
    I_\eps = I_{1,\eps} + I_{2,\eps} \leq C_1 \Big(\log\frac{1}{\eps}\Big)^{-2} \eps^{(p-N)\big( \frac{p}{2(p-1)}-1\big)}.   \label{ec:estimIeps}
\end{equation}
for some positive constant $C_1$ independent of $\eps$.
Now we split $J_\eps$:
\begin{align}
    J_\eps &= \sum_{i=1}^2\int_{B_\eps(a_i)\setminus B_{\eps^2}(a_i)} \big(\mu_2 V_2 - \eps_0 V_1\big)\abs{u_\eps}^p dx + \sum_{i=1}^2\int_{B_{\eps^{1/2}}(a_i)\setminus B_\eps(a_i)} \big(\mu_2 V_2 - \eps_0 V_1\big)\abs{u_\eps}^p dx  =: J_{1,\eps} + J_{2,\eps}.   \notag
\end{align}
Using \eqref{ec:eq22} and co-area formula, we get

\begin{align}
    J_{1,\eps} & = \sum_{i=1}^2\int_{B_\eps(a_i)\setminus B_{\eps^2}(a_i)} \big(\mu_2 V_2 - \eps_0 V_1\big)\abs{u_\eps}^p dx  \notag\\
    & < -\frac{\eps_0}{2} \sum_{i=1}^2\int_{B_\eps(a_i)\setminus B_{\eps^2}(a_i)} V_1 \theta_\eps^p\phi^p dx   \notag\\
    & \lesssim -\Big(\log\frac{1}{\eps}\Big)^{-p} \sum_{i=1}^2\int_{B_\eps(a_i)\setminus B_{\eps^2}(a_i)}  \Big( \log\frac{\abs{x-a_i}}{\eps^2} \Big)^p \abs{x-a_i}^{-p}\abs{x-a_i}^{p\be} dx   \notag\\
    & \simeq -\Big(\log\frac{1}{\eps}\Big)^{-p} 2\int_{\eps^2}^\eps  \Big(\log\frac{s}{\eps^2}\Big)^p s^{p(\be-1)+N-1} ds   \notag\\
    & \lesssim -\Big(\log\frac{1}{\eps}\Big)^{-p} \int_{\eps/2}^\eps  \Big(\log\frac{s}{\eps^2}\Big)^p s^{p(\be-1)+N-1} ds  \notag\\
    & \lesssim -\Big(\log\frac{1}{\eps}\Big)^{-p}\Big(\log \frac{1}{2\eps}\Big)^{p} \int_{\eps/2}^\eps s^{p(\be-1)+N-1}ds \notag\\
    & \lesssim -\eps^{(p-N)\big( \frac{p}{2(p-1)}-1 \big)}.   \notag
\end{align}
Similarly
\begin{align}
    J_{2,\eps} & = \sum_{i=1}^2\int_{B_{\eps^{1/2}}(a_i)\setminus B_\eps(a_i)} \big(\mu_2 V_2 - \eps_0 V_1\big)\abs{u_\eps}^p dx  \notag\\
    & < -\frac{\eps_0}{2} \sum_{i=1}^2\int_{B_{\eps^{1/2}}(a_i)\setminus B_\eps(a_i)} V_1 \theta_\eps^p\phi^p dx   \notag\\
    & \lesssim -\Big(\log\frac{1}{\eps}\Big)^{-p} \sum_{i=1}^2\int_{B_{\eps^{1/2}}(a_i)\setminus B_\eps(a_i)}  \Big(\log\frac{\eps}{\abs{x-a_i}^2}\Big)^p \abs{x-a_i}^{-p}\abs{x-a_i}^{p\be} dx   \notag\\
     & \simeq -\Big(\log\frac{1}{\eps}\Big)^{-p} 2\int_\eps^{\eps^{1/2}} \Big( \log\frac{\eps}{s^2} \Big)^p s^{p(\be-1)+N-1} ds       \notag\\
    & \lesssim -\Big(\log\frac{1}{\eps}\Big)^{-p} \int_\eps^{\eps^{2/3}} \Big( \log\frac{\eps}{s^2} \Big)^p s^{p(\be-1)+N-1} ds      \notag\\
    & \lesssim -\Big(\log\frac{1}{\eps}\Big)^{-p}\Big(\log\frac{1}{\eps^{1/3}}\Big)^{p} \int_\eps^{\eps^{2/3}} s^{p(\be-1)+N-1}ds      \notag\\
    & \lesssim -\eps^{(p-N)\big( \frac{p}{2(p-1)}-1 \big)}.   \notag
\end{align}
Combining the above two estimates, we get that 
\begin{equation}
    J_\eps = J_{1,\eps} + J_{2,\eps} < -C_2 \eps^{(p-N)\big( \frac{p}{2(p-1)}-1\big)},   \label{ec:estimJeps}
\end{equation}
where $C_2$ is a positive constant, independent of $\eps$.
By \eqref{ec:eq23}, \eqref{ec:estimIeps} and \eqref{ec:estimJeps} we obtain
\begin{align}
    L[u_\eps] & = I_\eps + J_\eps   \notag\\
    & <  C_1 \Big(\log\frac{1}{\eps}\Big)^{-2} \eps^{(p-N)\big( \frac{p}{2(p-1)}-1\big)} -C_2 \eps^{(p-N)\big( \frac{p}{2(p-1)}-1\big)}  \notag\\
    & = \eps^{(p-N)\big( \frac{p}{2(p-1)}-1\big)} \Big( C_1\Big(\log\frac{1}{\eps}\Big)^{-2} - C_2\Big) \stackrel{\eps\to 0}\longrightarrow 0_-.  \label{ec:contradictiefin}
\end{align}
Clearly, inequality \eqref{ec:contradictiefin} provides a contradiction with the assumption \eqref{ec:contradictie}. The proof of Theorem \ref{th2} is finished.


\section*{Appendix: Proof of Proposition 2.2}

Recall that, by \eqref{ec:potentialfinal},
\begin{multline} 
	V    = \frac{p-1}{4}\Bigg(\frac{N-p}{p-1}\Bigg)^p \frac{\abs{a_1-a_2}^2\abs{x-a}^{p-2}}{\abs{x-a_1}^p\abs{x-a_2}^p} 
  +   \frac{p-2}{2}\Bigg( \frac{N-p}{p-1} \Bigg)^{p-1} \times \\
  \times   \frac{\abs{x-a}^{p-4}}{\abs{x-a_1}^p\abs{x-a_2}^p}\Big[\abs{x-a_1}^2\abs{x-a_2}^2 - \big( (x-a_1)\cdot(x-a_2) \big)^2\Big]. 
\end{multline}
It is clear, due to Cauchy-Schwarz inequality, that $V\geq 0$ for $p\in(2,N)$. Also, for $p=2$, $V=\mu_1 V_1$, which is positive.\\
Let $p\in(1,2)$. After some computations, we obtain V in the following form
\begin{align}
     V &= \frac{p-1}{8}\Bigg(\frac{N-p}{p-1}\Bigg)^{p-1}\frac{\abs
    {x-a}^{p-4}}{\abs{x-a_1}^p\abs{x-a_2}^p} \Bigg[ \frac{N-p}{2}\Big(\abs{x-a_1}^4+\abs{x-a_2}^4\Big)    \notag\\   
    &  + \Big(N+3p-8\Big)\abs{x-a_1}^2\abs{x-a_2}^2 + 2\Big(4-p-N\Big)\Big((x-a_1)\cdot(x-a_2)\Big)^2 \Bigg]    \notag
\end{align}
Applying the Cauchy-Schwarz inequality, we get
\begin{align}
    V & \geq \frac{p-1}{8}\Bigg(\frac{N-p}{p-1}\Bigg)^{p-1}\frac{\abs
    {x-a}^{p-4}}{\abs{x-a_1}^p\abs{x-a_2}^p}\times   \notag\\
    & \times \Bigg[ \frac{N-p}{2}\Big(\abs{x-a_1}^4+\abs{x-a_2}^4\Big) - (N-p)\abs{x-a_1}^2\abs{x-a_2}^2 \Bigg]   \notag\\
    & = \frac{p-1}{16}\Bigg(\frac{N-p}{p-1}\Bigg)^p\frac{\abs{x-a}^{p-4}}{\abs{x-a_1}^p\abs{x-a_2}^p} \Big(\abs{x-a_1}^2- \abs{x-a_2}^2\Big)^2.    \notag
\end{align}
It is clear that the above quantity is positive for any $1<p<2$. The proof is done. 
\hfill $\square$

\end{document}